\documentclass[12pt]{amsart} 
\usepackage{amsfonts,amsmath,amssymb,amsthm} 
\usepackage{graphicx} 
 \newtheorem{theorem}{Theorem}

\newtheorem{Exe}{Example} 
 
\newtheorem{remark}{Remark}

\begin{document}

\title[Iterated function systems with a given continuous stationary distribution]{Iterated function systems with a given continuous stationary distribution
}

\author{ \"{O}rjan Stenflo}
\email{stenflo@math.uu.se}
 \address{Department of Mathematics,
Uppsala University,
751 06 Uppsala,
Sweden
}

\begin{abstract}  
For any continuous probability measure $\mu$ on ${\mathbb R}$  we construct an IFS with probabilities having $\mu$ as its unique  measure-attractor.
\end{abstract} 

\maketitle

\section{Introduction}

 In 1981 Hutchinson \cite{Hutchinson81} presented a theory  of fractals and measures supported on fractals based on iterations of functions. 

Let $\{ {\mathbb R}^d; f_i,p_i,\ i=1,...,n  \}$
be an iterated function system  with probabilities (IFSp). That is, $f_i: {\mathbb R}^d \rightarrow {\mathbb R}^d$, $i=1,...,n$, are functions and $p_i$ are associated non-negative numbers with $\sum_{i=1}^n p_i=1$.  If the maps
 $f_i: {\mathbb R}^d \rightarrow  {\mathbb R}^d $ are contractions, i.e.\  if there exists a constant $c<1$ such that $|f_i(x) - f_i(y)| \leq c |x-y|$, for all $x,y \in {\mathbb R}^d$, then there exists a unique nonempty compact set $A$ satisfying 
\begin{equation} \label{Attr}
A = \cup_{i=1}^n f_i(A)= 
\{ \lim_{k \rightarrow \infty } 
f_{i_1} \circ f_{i_2} \cdots \circ f_{i_k}(x);\ \ \ i_1 i_2 i_3...\in \{1,...,n\}^\mathbb{N} \},
\end{equation}
for any $x \in {\mathbb R}^d$, and a unique
  probability measure $\mu$, supported on $A$, satisfying the invariance equation
\begin{equation} \label{valborg}
\mu( \cdot ) = \sum_{i=1}^n p_i \mu( f_i^{-1} (\cdot)),
\end{equation}
see Hutchinson \cite{Hutchinson81}.
The set $A$ is sometimes called  the set-attractor, and $\mu$ the
  measure-attractor of the IFSp.

The set-attractor $A$ will have a  self-repeating ``fractal'' appearance 
if all maps $f_i$ are similitudes, and the sets, $f_i(A)$, $i=1,...,n$, do not overlap.  
This leads to the intuition to  regard the set-attractor $A$ in \eqref{Attr} as 
 being built up by $n$ (in general  overlapping and  
heavily distorted) ``copies'' of itself, and the measure-attractor as a ``greyscale colouring'' of the set-attractor.  
(Note that the probabilities $p_i$ play no role in the definition of $A$.)

In general we can not expect to have a unique set-attractor if the  IFS-maps are not assumed to be contractions or more generally if the limits
$\lim_{k \rightarrow \infty} f_{i_1} \circ f_{i_2} \cdots \circ f_{i_k}(x)$
do not exist, with the limit being independent of $x$, for {\em all}  $i_1 i_2 i_3...\in \{1,...,n\}^\mathbb{N}$, but unique measure-attractors exist if the limits
\begin{equation} \label{0630}
\widehat{Z}^F(i_1 i_2 ...):=
\lim_{k \rightarrow \infty} f_{i_1} \circ f_{i_2} \cdots \circ f_{i_k}(x)
\end{equation}
exist (with the limit being independent of $x$) for {\em almost all}  $i_1 i_2 i_3...\in \{1,...,n\}^\mathbb{N}$.
(Indeed, if the limit in \eqref{0630} exists a.s.\ then 
$\widehat{Z}^F$ may be regarded as a random variable, and its distribution
  $\mu(\cdot):=P(\widehat{Z}^F \in \cdot)$, is then the unique solution to
\eqref{valborg}.)

The theory of IFSp has a long pre-history within the theory of Markov
 chains, starting already with papers in the 30th by D\"oblin and others.
Let $\{X_k\}_{k=0}^\infty$ be the Markov chain obtained by random (independent)
 iterations with the functions, $f_i$, chosen with the 
corresponding probabilities, $p_i$. That is, let  $\{X_k\}$ be defined
 recursively by
$$
 X_{k+1}= f_{I_{k+1}}( X_k),\ k \geq 0,
$$
where $\{I_k\}_{k=1}^\infty$ is a sequence of independent random variables with $P(I_k=i)=p_i$, independent of $X_0$, where $X_0$ is some given random variable.
(It is well-know that any Markov chain  $\{X_k\}$ (with values in $
 \mathbb{R}^d$) can be expressed in the form  $X_{k+1}= g( X_k, Y_{k+1})$ where
$g: \mathbb{R}^d \times [0,1] \rightarrow  \mathbb{R}^d$ is a measurable function and  $\{Y_k\}_{k=1}^\infty$ is a sequence of independent random variables
uniformly distributed on the unit interval, see e.g.\  Kifer \cite{Kifer86}.)

If an IFSp has a unique measure-attractor, $\mu$, then $\mu$ is the 
 unique stationary distribution of $\{X_k\}$, i.e.\
 $\mu$ is the unique probability measure with the property that
 if $X_0$ is $\mu$-distributed,  then $\{X_k\}$ will be a (strictly)   stationary (and ergodic) stochastic process, see e.g.\ Elton \cite{Elton87}.
Therefore a unique measure-attractor can alternatively also be called a unique stationary distribution.

 Under standard average contraction conditions   it follows that
(\ref{0630}) holds a.s., and
the distribution of $X_k$ converges weakly to $\mu$ (with exponential rate quantified  e.g.\ by the Prokhorov metric for arbitrary distributions of the initial random variable $X_0$). Moreover the empirical distribution along  trajectories of $\{X_k\}$  converges weakly to $\mu $ a.s., 
and $\{X_k\}$   obeys a central limit theorem.
  See e.g.\ Barnsley et al.\ \cite{Barnsleyetal08}, Diaconis and Freedman \cite{DiaconisFreedman99}, and Stenflo \cite{Stenflo11}
 for details and further results. These papers also contains surveys of the
 literature.

\subsection{The inverse problem}

The inverse problem is to, given a probability measure $\mu$, find an IFSp having $\mu$ as its unique measure-attractor. This problem is of importance in e.g.\ image coding where the image, represented by a probability measure, can be encoded by the parameters in a corresponding IFSp in the affirmative cases, see e.g.\ Barnsley \cite{Barnsley93}. For an encoding to be practically 
useful it
 needs to involve few parameters and the
 distribution of $X_k$ needs to converge quickly 
to equilibrium  (a property ensured by average contractivity
properties of the functions in the IFSp) for arbitrary initial distributions of $X_0$.     


It is possible to construct solutions to the inverse problem  in some very 
 particular cases using Barnsley's  ``collage theorem'', see \cite{Barnsley93} containing exciting examples of e.g.\ ferns and clouds (interpreted as probability measures on $\mathbb{R}^2$) and their IFSp encodings, 
but typically it is very hard to even find approximate solutions to the inverse problem for general probability measures on ${\mathbb R}^d$.

In this paper we present a (strikingly simple) solution to
  the inverse problem for 
continuous probability measures on ${\mathbb R}$.

\section{Main result}

In order to present our solution to the inverse problem for 
continuous probability measures on ${\mathbb R}$,  recall the following basic facts
 used in the theory of random number generation; 
 
Let $\mu$ be a  probability measure on ${\mathbb R}$, and
 let $F(x)= \mu((-\infty,x])$ denote its distribution function.
The generalised inverse distribution function is defined by
$$F^{-1}(u)= \inf_{x \in {\mathbb R}} \{ F(x) \geq u \},\ \ \ 0  \leq u  \leq 1,$$ and satisfies
$F^{-1}(F(x)) \leq x$ and $ F(F^{-1}(u)) \geq u$ and therefore
$$ F^{-1}(u) \leq x \hspace{5mm} \text{ if and only if } \hspace{5mm}
u \leq F(x).$$
From this it follows that if $U \in U(0,1)$, i.e.\ if $U$ is a random variable
 uniformly distributed on the unit interval, then  $F^{-1}(U)$ is a $\mu$-distributed random variable. 
This basic property reduces the problem of simulating from an arbitrary distribution on ${\mathbb R}$, to the problem of simulating uniform random numbers on the unit interval.
 
We say that $\mu$ is continuous if $F$ is continuous.
Note that $\mu(\{x\})=0$ for any $x \in {\mathbb R}$ for continuous probability measures in contrast with discrete probability measures where $\sum_{x \in S} \mu(\{x\})=1$ for some countable set $S$.

If $\mu$ is continuous  then $ F(F^{-1}(u)) = u$, for $0 <u<1$.
This property is crucial for the following theorem;

\begin{theorem} \label{110622} 
A continuous distribution, $\mu$, on $\mathbb{R}$ with
 distribution function, $F$, is the measure-attractor of the 
IFS with monotone  maps   
 $f_i(x):= F^{-1} \circ u_i \circ  F(x) $, for any $x$ with $F(x) >0$, 
 and probabilities $p_i=1/n$, where
  $u_i(u)=u/n + (i-1)/n$, $0 \leq u \leq 1$,  $i=1,2,...,n$, for any $n \geq 2$.
\end{theorem}

\begin{proof}
The Markov chain generated by  
$u_i(x)=x/n + (i-1)/n$, $i=1,2,...,n$, chosen
 with equal probabilities has the uniform distribution on the unit interval as its  unique stationary distribution.
 That is, if $\{I_k\}_{k \geq 1}$ is a sequence of independent random variables, uniformly distributed on $\{1,2,...,n\}$, then
\begin{equation}
Z_k^U(x)=u_{I_k} \circ \cdots \circ  u_{I_1} (x),\ \  \ 
Z_0^U(x)=x
\end{equation} 
 is a Markov chain starting at $x \in [0,1]$ having the uniform distribution as its unique stationary distribution.
 This can be seen by  observing that $Z_k^U(x)$ has the same distribution as the reversed iterates
 \begin{equation} \label{reversed}
\widehat{Z}_k^U(x)=u_{I_1} \circ \cdots \circ  u_{I_k} (x),\ \  \ 
\widehat{Z}_0^U(x)=x,
\end{equation}
 for any fixed $k $, and the reversed iterates $\widehat{Z}_k^U(x)$ converges almost surely to the $U(0,1)$-distributed random variable, $\widehat{Z}^U$,
 where the $k:th$ digit in the base $n$ expansion of $\widehat{Z}^U$ is given by $I_k-1$. 
 
If $\widehat{Z}^F$ denotes the limit of the reversed iterates of the system with $f_i$  chosen with probability $1/n$, then
\begin{eqnarray} \label{0826}
\widehat{Z}^F&:=&
 \lim_{k \rightarrow \infty} \widehat{Z}_k^F(x):= 
 \lim_{k \rightarrow \infty} 
 f_{I_1}   \circ \cdots \circ   f_{I_k}   (x)
 \nonumber \\
&=& 
 \lim_{k \rightarrow \infty}
 F^{-1} \circ u_{I_1} \circ  F 
 \circ
 F^{-1} \circ u_{I_2} \circ  F 
 \circ
 F^{-1} \circ u_{I_k} \circ  F (x) \nonumber \\
 &=& 
 \lim_{k \rightarrow \infty} F^{-1} \widehat{Z}_k^U ( F(x))= F^{-1} ( \widehat{Z}^U) \ \ \ a.s.,
 \end{eqnarray}
where the last equality holds since $F^{-1}(x) $ 
is non-decreasing, and since a
monotone function can have at most a countable set of discontinuity points in its domain, it follows that $F^{-1}(x) $ 
is continuous for a.a. $x \in [0,1]$ w.r.t. to the Lebesgue measure.

From the above it follows that
$$ P( \widehat{Z}^F \leq y ) = P( F^{-1} (\widehat{Z}^U)  \leq y)=P( \widehat{Z}^U 
\leq F(y))=F(y).$$ 
\end{proof}

\begin{remark}
If $X$ is a continuous  $\mu$-distributed random variable, then $ F(F^{-1}(u)) = u$, so $F(X) \in U(0,1)$. This contrasts the case when $X$ is discrete where $F(X)$ will also be discrete, so we cannot expect Theorem \ref{110622} to generalise to discrete distributions.

If an IFS $\{ {\mathbb R}, f_i, p_i, i=1,...,n \}$, has a continuous measure-attractor $\mu$ being the distribution of the a.s.\ limit of the reversed iterates, and the 
distribution function $F$ of $\mu$ satisfies $F^{-1} (F(x))=x$, for any $x \in {\mathbb R}$, with $0<F(x)<1$, then, similarly, the IFS
$\{ {[0,1]}, u_i, p_i, i=1,...,n \}$, with
$u_i(u):= F \circ f_i \circ  F^{-1}(u)$, $0 <u <1$, has the $U(0,1)$-distribution as its unique stationary distribution.
This is the case for absolutely continuous probability distributions $\mu$ if $F$ is strictly increasing.

\end{remark}

\begin{remark}
From Theorem \ref{110622} it follows that any continuous probability distribution on ${\mathbb R}$  can be approximated by the empirical distribution of 
 a Markov chain $\{ X_k \}$  on ${\mathbb R}$ generated by an IFSp 
 with trivial "randomness" generated by e.g.\ by a coin or a dice.

\end{remark}

\begin{remark}
 Theorem \ref{110622} may be used to represent a continuous  probability measure $\mu$ on ${\mathbb R}$ by the functions  suggested  in the theorem.
Note that there  exist many iterated function systems with probabilities generating the same Markov chain, see e.g.\ Stenflo \cite{Stenflo01}, so
 in particular it follows than an IFSp representation of a continuous probability measure on ${\mathbb R}$ is not unique. 
 The given IFSp representation suggested by Theorem \ref{110622} (for a given $n \geq 2$)  is good in the sense that the generated Markov chain converges quickly to the given equilibrium making it possible to quickly simulate it. 
If the suggested IFSp representation cannot be  described in terms of
 few parameters then it might make sense to consider 
 an approximate representation by approximating the IFS functions
 with functions described by few parameters  e.g.\ by using Taylor expansions.

\end{remark}

\begin{remark}
From Theorem \ref{110622} it follows
 that if $\mu$ is a continuous probability measure on ${\mathbb R} $ being the measure-attractor of 
$\{ {\mathbb R}; f_i,p_i,\ i=1,...,n  \}$, with $p_i \neq 1/n$ for some $n$,
then there exists another IFSp with uniform probabilities having $\mu$ as its measure-attractor.

\end{remark}

\begin{Exe}

Suppose $F$ is a distribution function satisfying
$$F(1-x)=1-F(x),$$
and
$$F(x)/2= F(a x+b), {\text{ for all }} 0 \leq x \leq 1,$$ 
where $0 \leq b \leq 1/2$, $0 \leq a+b \leq 1/2$, and $a \neq 0$.

Then
$$F(x)/2+1/2 =  1- F(1-x)/2=
1- F( a(1-x)+b)=F(a x+ 1-a-b).$$ 

Thus random iterations with the maps
$f_1(x)=  ax +b$, and  $f_2(x)=ax + 1-a-b$ chosen with equal probabilities generates a Markov chain with stationary distribution $\mu$ having distribution function $F$.

The case $a=1/3$ and $b=0$  corresponds to $F$ being the distribution function of the uniform probability measure on the  middle-third Cantor set (the Devil's staircase).

\begin{center} 
\resizebox{!}{40mm}{\includegraphics{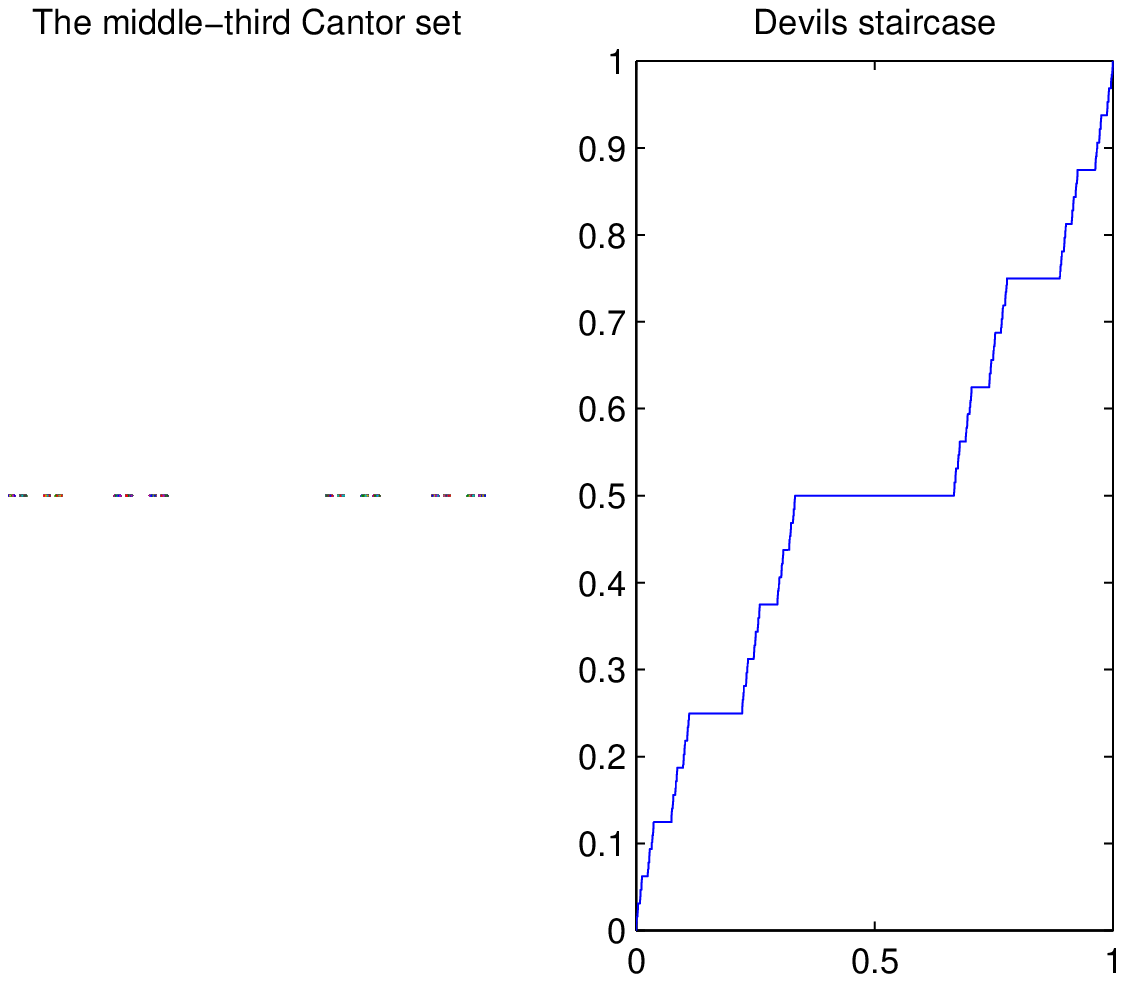}} 
\end{center}
\hspace{23mm}\parbox{12cm}{{\em The Cantor set
is the set-attractor of the  IFSp \linebreak[4] $\{ {\mathbb R}; f_1(x)=x/3, f_2(x)=x/3+2/3, p_1=1/2, p_2=1/2 \}$ 
and the distribution function of its measure-attractor (the uniform distribution on the Cantor set)  is an  increasing continuous function with zero derivative almost everywhere, with $F(0)=0$ and $F(1)=1$  popularly known as the ``Devil's staircase''. 
  }}

 \end{Exe}

\begin{Exe} Let $\mu$ be the probability measure 
with triangular density function
$$ d(x) = \begin{cases}
x & 0 \leq x \leq 1 \\
2-x & 1 \leq x \leq 2
  \end{cases}. $$
Then $\mu$ is the unique stationary distribution of the Markov chain generated by random iteration with the functions 
$$ f_1(x) = \begin{cases}
\frac{x}{\sqrt{2}} & 0 \leq x \leq 1 \\
\sqrt{ 2x- \frac{x^2}{2}-1 } & 1 \leq x \leq 2,
\end{cases} $$
and
$$ f_2(x) = \begin{cases}
2-\sqrt{ 1- \frac{x^2}{2} } & 0 \leq x \leq 1 \\
2- \sqrt{ 2 - 2x + \frac{x^2}{2} } & 1 \leq x \leq 2,
\end{cases}, $$
chosen uniformly at random.
\begin{center} 
\resizebox{!}{60mm}{\includegraphics{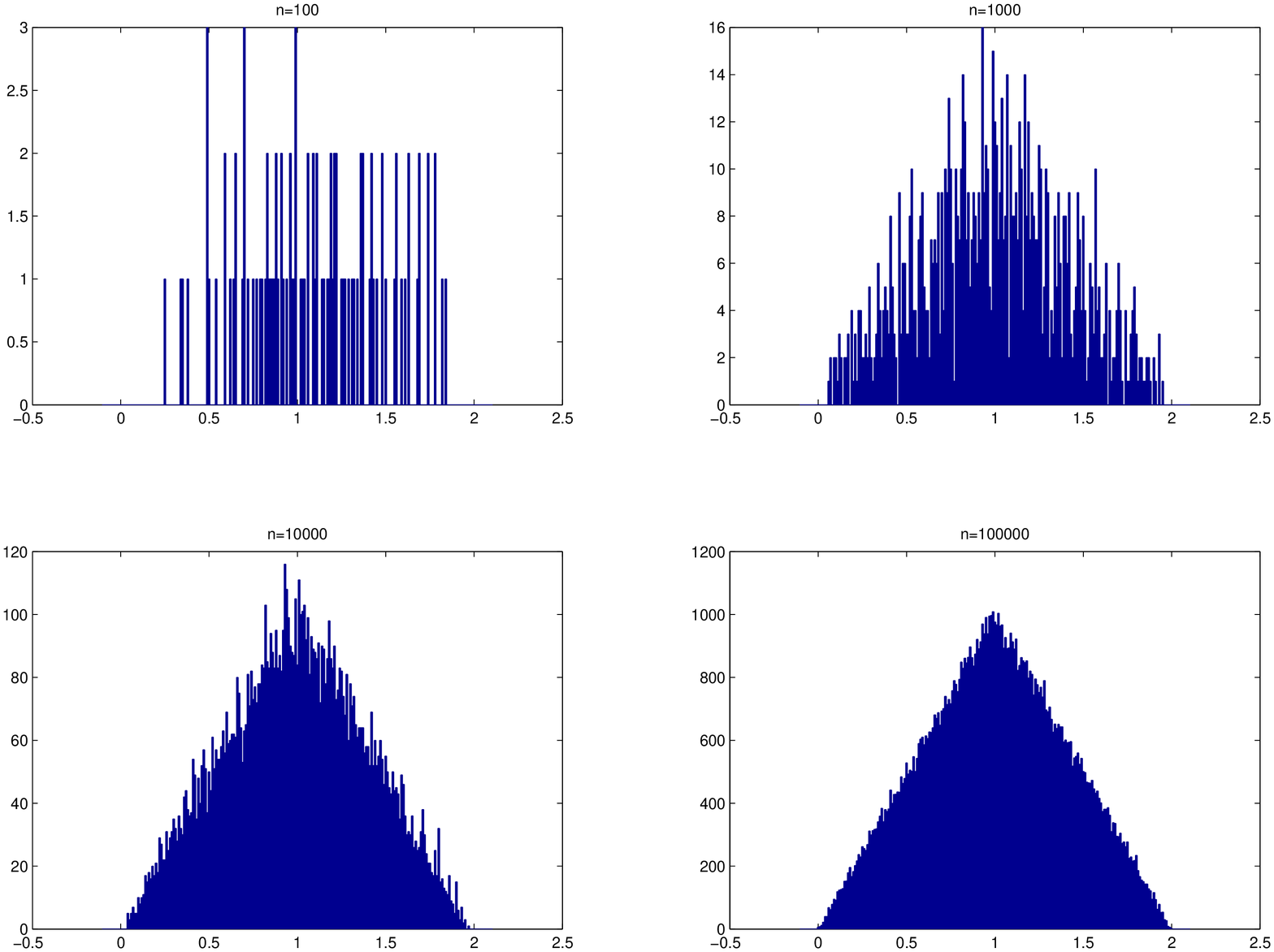}} 
\end{center}
\hspace{23mm}\parbox{12cm}{{\em  Histograms of the first $n$  points in a simulated random trajectory of the Markov chain. The empirical distribution along a trajectory converges weakly to the stationary triangular-distribution with probability one. 
  }}

\end{Exe}

\begin{Exe}
The distribution function for the exponential distribution with
 expected value $\mu= \lambda^{-1}$, $\lambda >0$, satisfies
$F(x)=1-e^{- \lambda x}$, $x \geq 0$.
 A  Markov chain generated by  
random iterations with the two maps $f_1=f_1^\mu$ and $f_2=f_2^\mu$
 defined as in 
Theorem \ref{110622} has the exponential distribution with expected value $\mu$ as its stationary distribution.
We can construct interesting ``new'' distributions by altering such Markov chains in various ways, e.g.\ by altering the application of two IFSs corresponding to different parameter values. 
A result of such a construction is shown in the figure.

\begin{center} 
\resizebox{!}{60mm}{\includegraphics{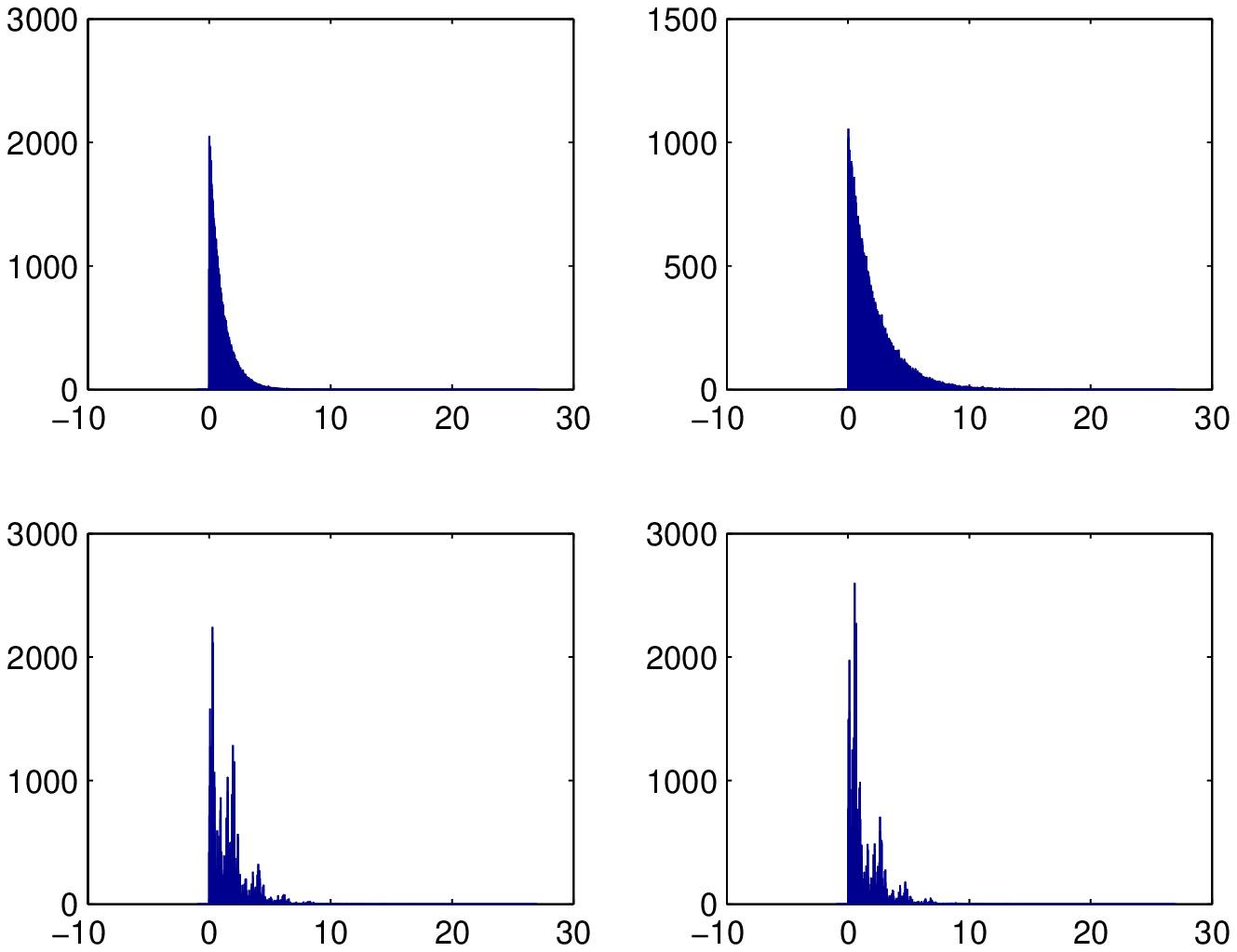}} \label{Exponential}
\end{center}
\hspace{23mm}\parbox{12cm}{{\em  
 
 }}
 
The upper figures are histograms of the first 200000 points
 in simulations of a trajectory of a Markov chain generated by  
random iterations with the two maps $f_1=f_1^\mu$ and $f_2=f_2^\mu$
 defined as in 
Theorem \ref{110622} 
 corresponding to the choices $\mu=1$ in the left hand figure and 
 $\mu=2$ in the righthand figure respectively.

The lower figures are histograms corresponding to trajectories of  Markov chains formed by random iterations with the maps
 $g_1(x)=
 f_1^2 (f_1^1(x))$,  $g_2(x)=
 f_1^2 (f_2^1(x))$,  $g_3(x)=
 f_2^2 (f_1^1(x))$,  $g_4(x)=
 f_2^2 (f_2^1(x))$ and  $h_1(x)=
 f_1^1 (f_1^2(x))$,  $h_2(x)=
 f_1^1 (f_2^2(x))$,  $h_3(x)=
 f_2^1 (f_1^2(x))$,  $h_4(x)=
 f_2^1 (f_2^2(x))$ 
respectively, where in both cases  the functions are chosen uniformly at random. 
\end{Exe}

\begin{remark}
The distributions constructed in the lower figures in the example above are 
$1-$variable mixtures of the exponential distributions with expected values  $\mu=1$, and  $\mu=2$ respectively.
We can, more generally,  for any integer $V \geq 1$, generate $V-$variable mixtures between continuous distributions.
See Barnsley et al. \cite{Barnsleyetal05} and \cite{Barnsleyetal08} for more on the theory of  
   $V-$variable  sets and measures.
\end{remark}

\end{document}